\documentclass[12pt, reqno]{amsart}
\usepackage[margin=3.6cm]{geometry}

\usepackage[utf8]{inputenc}
\usepackage{amsmath, amsthm,amssymb,color}
\usepackage{amsfonts,dsfont,enumitem}

\newtheorem{Theorem}{Theorem}
\newtheorem{prop}{Proposition}[section]
\newtheorem{cor}[prop]{Corollary}
\newtheorem{Lemma}[prop]{Lemma}

\theoremstyle{definition}

\newtheorem{rem}{Remark}[section]
\newtheorem{defn}[prop]{Definition} 
\newtheorem{assu}{Assumptions}[section]

\newcommand{\C}{\mathbb{C}}

\newcommand{\Mcal}{\mathcal{M}}

\newcommand{\Vcal}{\mathcal{V}}
\newcommand{\Wcal}{\mathcal{W}}
\newcommand{\Acal}{\mathcal{A}}
\newcommand{\bfc}{\mathbf{c}}

\newcommand{\deq}{\overset{\mathrm{def}}{=}}
\newcommand{\rmax}{\mathrm{max}}
\newcommand{\rmin}{\mathrm{min}}
\newcommand{\Emin}{E_\rmin}
\newcommand{\Emax}{E_\rmax}
\newcommand{\Fcal}{\mathcal{F}}
\newcommand{\tilh}{{\tilde h}}
\newcommand{\Scal}{\mathcal{S}}
\newcommand{\Lcal}{\mathcal{L}}

\newcommand{\tbet}{{\tilde{\beta}}}
\newcommand{\Op}{{\mathrm{Op}}}

\newcommand{\Z}{\mathbb{Z}}

\newcommand{\N}{\mathbb{N}}
\newcommand{\R}{\mathbb{R}}
\newcommand{\eps}{\varepsilon}
\newcommand{\dis}{\displaystyle}
\newcommand{\alp}{\alpha}

\newcommand{\un}{\mathds{1}}

\renewcommand{\leq}{\leqslant}
\renewcommand{\geq}{\geqslant}
\renewcommand{\Im}{\mathrm{Im}}

\newcommand{\defeq}{\stackrel{\rm{def}}{=}}
\newcommand{\spec}{\operatorname{spec}}
\newcommand{\Kcal}{\mathcal{K}}
\newcommand{\tv}[2]{\underset{{#1}\rightarrow {#2}}{\longrightarrow}}

  \title{Eigenvalue spacing for 1D singular Schr\"odinger operators }
\author[L. Hillairet]
{Luc Hillairet}
\email{luc.hillairet@math.univ-orleans.fr}
\address{Institut Denis Poisson, Universit\'e d'Orl\'eans,\\
Orl\'eans, France}
  
  \author[J.L. Marzuola]
{Jeremy L. Marzuola}
\email{marzuola@math.unc.edu}

\address{Mathematics Department, University of North Carolina \\
CB\#3255, Phillips Hall, Chapel Hill, NC USA}
 
\begin{document}
  
  \begin{abstract}
    The aim of this paper is to provide uniform estimates for the eigenvalue spacings of
    one-dimensional semiclassical Schr\"odinger operators with singular potentials on the half-line.  We introduce a new development of semiclassical measures related to families of Schr\"odinger operators that provides a means of establishing uniform non-concentration estimates within that class of operators.  This dramatically simplifies analysis that would typically require detailed WKB expansions near the turning point, near the singular point and several gluing type results to connect various regions in the domain. 
\end{abstract}
\maketitle

\section{Introduction}

We consider a (self-adjoint) one dimensional semiclassical Schr\"odinger operator 
\[
P_h u \,=\, -h^2 u''\,+\, V(x) u
\] 
that is defined on the half-line $I=[0,+\infty).$ The potential $V$ is defined by
$x\mapsto x_+^\gamma W(x)$ for some $\gamma >0$ and a smooth, positive $W$. 
We will be interested in the eigenvalue equation 
\begin{equation}\label{eq:eigeq}
P_h u_h=E_h u_h,
\end{equation} 
for an energy $E_h$ in a certain regime that is a, possibly $h$-dependent, compact interval $K_h\subset \R$ that
we call the energy window. If the spectrum of $P_h$ is discrete in $K_h$ we define, for $E$ in $\spec P_h$, 
\[
  d_h(E) \defeq \inf \{ |E-\tilde{E}|,~ \tilde{E}\in \spec P_h, \tilde{E} \neq E\},  
\]
and we aim at giving lower bounds on $d_h(E)$ as uniform as possible.

Studying Schr\"odinger operators is a standard problem in spectral theory and many results on eigenvalues
and eigenfunctions can be extracted from the literature on Sturm-Liouville problems and semiclassical analysis
(Titchmarsch \cite{titchmarsh1946eigenfunction}, Olver \cite{olver}, H\"ormander \cite{Hor-v1,Hor-v2,Hormander3,Hormander4},
Maslov \cite{Maslov72}, 
Helffer-Robert \cite{helffer1983calcul}, Dimassi-Sj\"ostrand \cite{DiSj}, Zworski \cite{zworski2012semiclassical}).

In particular, Bohr-Sommerfeld rules for smooth potentials in the semiclassical literature imply that, for a sequence of
eigenvalues $(E_h)_{h>0}$ that converges to a non-critical energy $E_0$ with a connected energy surface,
then the spacing is of order $h$ (see Section
10.5 in \cite{Bender_Orszag78} or \cite{Cdv_BS05, Yafaev11} for instance).
In most cases, semiclassical techniques allow one to work in any dimension but, often, only for smooth potentials.

Singular potentials have also been studied (see among others \cite{LaiRobert79, Berry82, Chr_AIF_15}).
Often, the "bottom-of-the well" regime is considered, i.e. when $E_h$ goes to $0$ at a certain rate.
The latter rate can be obtained by a scaling argument by deciding for which power $\alpha$ the change of variables $x\leftarrow h^\alpha x$
transforms the problem into a non-semiclassical second order differential equation. It can then be proved that the $k$-th eigenvalue of $P_h$ behaves like 
$a_kh^{\frac{2\gamma}{\gamma +2}}$ from which we infer that the spacing in this regime is also of order $h^{\frac{2\gamma}{\gamma+2}}$
(see \cite{friedlander2009spectrum}, and also \cite{Simon_lowlying_83} for a much more complete study of the bottom of the well for quadratic potentials, or \cite{Bony_Popoff19} for even more degenerate situations). We also advertise
the recent paper \cite{GW} that lays the foundations for a systematic semiclassical study of a class of singular potentials.

The intermediate regime, which is neither the non-critical energies nor the bottom of the well is known in the semiclassical literature
as \textit{semi-excited states} and has been initiated by Sj\"ostrand \cite{Sjostrand_semi92}.  

Our main result is stated as follows and can be seen as an estimate unifying all the preceding regimes.

\begin{Theorem}
\label{thm:main}
    Assume that $\gamma>0$ and $W$ is smooth and positive on $[0,+\infty)$. Let $V=x^\gamma W$ and $P_h$ the Dirichlet or Neumann realization of
    $-h^2u''+V$ on $[0,+\infty)$. If $\liminf_{x\rightarrow +\infty} V(x) >0$, there exist $M, h_0, \bfc>0$ such that
    \begin{enumerate}
    \item For all $h\leq h_0$, $\spec P_h \cap [0,M]$ is purely discrete,
    \item For any $h\leq h_0$ and any $E$ in $\spec P_h \cap [0,M]$,
      \[
d_h(E) \,\geq \, \bfc h\cdot E^{\frac{\gamma-2}{2\gamma}}.
\]
\end{enumerate}
\end{Theorem}

Such a theorem is actually equivalent to answering the following question: consider a sequence $(E_h)_{h\geq 0}$ going to some
limit $E_0$ as $h$ goes to $0$ and study the behavior of the sequence $(d_h(E_h))_{h\geq 0}$. When $E_0$ is non-critical then
our result recovers the usual order $h$ separation.
This is completely standard if $\gamma$ is an integer, for, in that case, the potential is smooth
and the full semiclassical machinery can be used. If $\gamma$ is not an integer, the energy surface is not smooth anymore and it must be proved that
the singularity is not strong enough to perturb the order $h$ spacing of eigenvalues.
This can perhaps be done by rather soft techniques such as some Dirichlet-Neumann bracketing argument. We have chosen a different, also
well known, technique that relies in estimating how fast the semiclassical Cauchy datum of the fundamental solution at $x=0$ winds around
the origin. We will observe that this winding is related to non-concentration
at the singular point. One motivation for studying this kind of potential comes from the adiabatic ansatz in a stadium-like
billiard (see \cite{hillairet2012nonconcentration}). In the latter, the potentials that come up are of the form
$x\mapsto x_+^\gamma W(x)$ on the half-line $[-B,+\infty)$ for $B > 0$ and the eigenvalue problem can be restated as a gluing
problem that involves the fundamental solution on the half-line that we study here.
We also point out that our assumptions imply that the energy surface is connected so that no tunneling effect
has to be taken into account (see \cite{HeSjo_mwI84, Mart_Rou88} for the more delicate case involving such tunneling effects).

\subsection*{Organization of the paper}
In Section \ref{s:bottom}, we will treat the bottom of the well regime. All the results of this section can be found in the literature but
we will outline a proof so as to make this paper self-contained.

In Section \ref{sec:spacingviaconcentration}, we will first give a general strategy of proof to obtain the eigenvalue spacing for
$1D$ Schr\"odinger operators. Our assumptions will imply that the vector space of $L^2$ solutions to
$(P_h-E)u =0$ is one dimensional so that the eigenvalue spacing will follow from the study of $G_h(\cdot\, ;\,E)$
which is a conveniently normalized solution
to this equation. We will in particular observe that the winding argument that leads to $h$-spacing
in the non-critical case can be reduced to a concentration
estimate. We will also show that, using an energy-dependent scaling,
the latter estimate in the intermediate regime can be obtained from estimates in the non-critical regime that
are uniform with respect to the potential. 

This will lead us to standard problems in semiclassical analysis with the twist that the potential is not fixed but lives in
some set $\Vcal$ of functions. In Sections \ref{sec:agmon} and \ref{sec:measures} we tackle the problems of exponential decay and
semiclassical measures from this point of view and we prove essentially that the usual statements remain true with constants that are
uniform in $\Vcal$ provided the latter set exhibits some compactness.
These two sections address the way the function $G_h(\cdot \,;\,E)$ may concentrate
in the classically not allowed region and near the turning point so that the singularity at $0$ actually does not play any role.
It then remains to address the classically allowed region and this will be done in Section \ref{sec:WKB}
in which we will combine WKB expansions with a Volterra type approach. We will need only the first order approximation but 
we will have to treat the cases $\gamma<1$ and $\gamma \geq 1$ separately. In the latter case, the first order correction is
of magnitude $h$ and we obtain directly a WKB-approximation for $G$ down to $x=0$. when $\gamma <1$, we will have to perform
a matching at $x=h$ and the first order correction will be of magnitude $h^\gamma$.

In the final section, we will patch all the different regimes to obtain the proof of Theorem \ref{thm:main}.

\subsection*{Acknowledgments} The authors are grateful to Jared Wunsch for helpful conversations. This work initiated when the second author
visited the first for an extended stay as a professeur invit\'e at the Universit\'e d'Orl\'eans and also benefited from the invitation of the first author
to the UNC at Chapel Hill. The authors thank both institutions. J.L.M. acknowledges supports from the NSF through NSF CAREER Grant DMS-1352353 and NSF Grant
DMS-1909035, and L.H. acknowledges the support of the projet r\'egion APR-IA THESPEGE.

\section{Bottom of the well}
\label{s:bottom}
We recall that we consider the following Schr\"odinger equation
\[
  -h^2u''\,+\,V(x)u\,=\,Eu
\]
on the half-line $I=[0,+\infty)$ with either Dirichlet or Neumann boundary condition at $0$.
Before proceeding, we outline the conditions we will place on the potential $V$ moving forward.

\begin{assu}
\label{assu:bot}
The following properties of $V$ hold:
\begin{itemize}
\item The potential $V$ is smooth on $(0,\infty)$ and continuous on $I$.
\item $V(0)=0$ and there exist $\gamma >0$ and $W$ smooth on $[0,\infty)$ such
  that $\forall x>0,~ V(x) = x^\gamma W(x),~W(0)>0.$
\item There exists some $d>0$ such that
  \begin{gather*}
    \forall x\geq d,~V(x)\geq V(d) , \\
    \forall x\in (0,d],~V'(x)>0.
  \end{gather*}
\end{itemize}
\end{assu}

The latter assumption implies that for any $E< V(d)$, the energy surface
\[
  \big \{ (x,\xi)\in I\times \R,~ \xi^2+V(x)=E \big \} 
\]
is compact and connected. It follows that the spectrum of $P_h$ that lies below $V(d)$ consists of eigenvalues
of finite multiplicity (\cite{BirSol87}, Ch. 10.6 or \cite{RSv4}, Ch. XIII). Moreover, since the potential is of limit-point type near infinity
any eigenvalue in the preceding regime is necessarily simple (see Titchmarsch \cite{titchmarsh1946eigenfunction} or \cite{reed1978methods, gesztesy2006spectral, teschl2009mathematical}).
  
\begin{prop}
Under Assumptions \ref{assu:bot}, for any $M$, there exists $c>0$ and $h_0>0$ such that 
\[
\forall h\leq h_0,~ \forall E\in [0,Mh^{\frac{2\gamma}{\gamma
    +2}}]\cap \spec P_h,~ d_h(E)\geq c h^{\frac{2\gamma}{\gamma +2}}.
\]
\end{prop}

\begin{proof}

As this result is somewhat classical, we only outline the proof and refer the reader to \cite{friedlander2009spectrum, Hil_Kai} for complete details.

We use a scaling argument: set $\alpha =\frac{2}{\gamma+2}$ and
define $v_h(y)=u_h(h^{\alpha} y).$ The function $v_h$ is a solution to 
\[
-v''\,+\, \left( y^\gamma W(h^\alpha y) -e_h\right )v\,=\,0,
\]
where we have put $e_h=h^{-\frac{2\gamma}{\gamma+2}}E_h.$

One can then argue by min-max arguments that $e_h$ is close to an
eigenvalue of the operator 
\[
A(v)=-v''+ W(0)x^\gamma v, 
\]
with the same boundary condition.
In order to estimate the error term, we can
introduce the point $x_h = h^{{\alpha}-\eps}$ for some $\eps>0$, 
and then use the exponential decay for $x>x_h$ (see Section \ref{sec:agmon} below).

The eigenvalues of $A$ are spaced at order $1$ and this gives the result.
\end{proof}

\section{General strategy and scaling}\label{sec:spacingviaconcentration}

\subsection{Energy spacing and eigenfunction concentration}\hfill \\
It is well-known that the spacing between eigenvalues of a semiclassical 1-D Schrödinger operator
around non-critical energies with a connected energy surface is of order $h$. This fact is classically
derived from the Bohr-Sommerfeld quantization rules (cf section 10.5 in \cite{Bender_Orszag78} or \cite{Cdv_BS05, Yafaev11}).
We present here a strategy that, in the end,
relies on a concentration estimate for eigenfunctions. Showing this estimate uniformly with respect to the
potential will be the key to the spacing in the intermediate regime.

Consider the eigenvalue equation
\[
  (P_h-E)u_h\,=\,0,
\]
in which the potential satifies the same assumptions as before and $E$ is in some compact set $K \subset (0,V(d))$.
Since this equation is of limit point type near $\infty$, we know that 
\[
  \dim \left \{ u\in C^ \infty \cap L^2(0,+\infty),~ (P_h-E) u=0\right \} \,=\, 1,
\]
so that there is a unique solution $G(\cdot \,;\,E)$ that satisfies
\begin{gather*}
    (P_h-E)G_h(\cdot\,;\,E) \,=\,0\\
    \int_{0}^{+\infty} |G_h(x \,;\,E)|^2\, dx\,=\,1,\\
    \forall x\geq d,~G_h(x\,;\,E)>0.
\end{gather*}
It is also standard that the mapping $E \mapsto G_h(\cdot \,;\,E)$ is analytic from
$(0,V(d))$ into $L^2((0,+\infty))$. If we denote by $\dot{G}_h(\cdot;E)$ the derivative of $G$ with respect to $E$, then,
by differentiating the eigenvalue equation, we obtain
\begin{equation}\label{eq:defGdot}
  (P_h-E) \dot{G}(\cdot\,;\,E)\,=\,G(\cdot\,;\,E).
\end{equation}

We define
\[
  Z_h(E)\,=\, G_h(0\,;\,E)\,+\,ihG_h'(0,E),
\]
which we can write, in polar coordinates, as
\[
  Z_h(E)\,=\,|Z_h(E)|e^{i\theta_h(E)}
\]
in which $E\mapsto \theta_h(E)$ is analytic.

A straightforward computation yields
\[
  \begin{split}
    |Z_h(E)|^2 \, \dot{\theta_h}(E)\,&=\, \Im (\overline{Z_h(E)}\dot{Z}_h(E))\\
  &=\,\Wcal_0\left[G_h,\dot{G}_h\right],
  \end{split}
\]
where $\Wcal$ is the (semiclassical) Wronskian that is defined by
\[
  \Wcal_x\left[ f, g\right]\,=\,hf(x)g'(x)-hf'(x)g(x).
\]

The semiclassical Wronskian of $G_h$ and $\dot{G}_h$ can also be computed by multiplying equation
\eqref{eq:defGdot} by $G$, integrating, and making two integration by parts (the contribution of $+\infty$ vanishes since
the equation is of limit-point type there and both functions are $L^2$). We obtain
\[
  \begin{split}
    \int_0^{+\infty} G^2_h(x\,;\,E)\, dx&\,=\, h^2 \dot{G}_h'(0\,;\,E)G_h(0\,;\,E)-h^2\dot{G}_h(0\,;\,E)G_h'(0\,;\,E)\\
    &\,=\, h\Wcal_0\left[G_h,\dot{G}_h\right].
\end{split}
\]
Finally, we obtain
\[
 |Z_h(E)|^2\, \dot{\theta}_h (E)\,=\,\frac{1}{h} \int_0^{+\infty} G_h^2(x\,;\,E)\,dx. 
\]
That, we rewrite as
\begin{equation}
\label{Zthetaeq}
  dE\,=\,h |Z_h(E)|^2 d\theta,
\end{equation}
since $G$ is normalized.

Since being an eigenvalue is equivalent to asking that $G_h(\cdot \,;\,E)$ satisfies Dirichlet or Neumann
boundary condition at $0$, it follows that,
between two consecutive eigenvalues $\int d \theta \,=\, \pi$. We will thus get the spacing of order $h$ provided that there exists some
positive constant $c$ such that
\[
 \forall E \in K,~  |Z_h(E)|^2 \,\geq \, c. 
\]

One way to obtain this inequality is by using WKB expansions and semiclassical measures.
Indeed, the WKB expansion near $0$ will yield that, for some small $a$ 
\[
  |Z_h(E)|^2 \asymp \int_0^a G_h^2(x\,;\,E) \, dx
\]
and a semiclassical measure argument will yield that
\[
  \int_0^a G_h^2(x\,;\,E) \, dx \asymp \int_0^{+\infty} G_h^2(x\,;\,E) .
\]

Both these arguments are standard for a smooth potential for non-critical energies.
In the next section we show that an energy-dependent scaling allows to get the estimate
for the intermediate regime by following the same method of proof but for families of potentials.
Showing that the estimates are uniform with respect to both the potential and the energy will
finally yield Theorem \ref{thm:main}.
    
\subsection{Energy scaling for the intermediate region}\hfill \\
\label{s:intermediate}
Choose a sequence $(E_h,u_h)_{h\geq 0}$ that is a solution to \eqref{eq:eigeq} under the standing assumptions on $V$.
Recall that $E_h$ is in the intermediate regime if neither $E_h$ is non-critical, nor $E_h$ is in the bottom of the well regime.
Equivalently, this reads as 
\[
  E_h \tv{h}{0} 0,~~\mbox{and}~~ h^{-\frac{2\gamma}{\gamma+2}}E_h \tv{h}{0} +\infty.
\]

We perform a $E$-dependent scaling on the equation by setting $\tilde{v}_h(z)\,=\,\tilde{u}_h(E_h^{\frac{1}{\gamma}}z)$.
We obtain
  \begin{gather*}
    -h^2\tilde{u}_h'' \,+\, (x^\gamma W(x) -\tilde{E})\tilde{u}_h \,=\,0~\iff\\
    -h^2E_h^{-1-\frac{2}{\gamma}}{\tilde{v}_h}''\,+\, (z^\gamma W(E_h^{\frac{1}{\gamma}}z)-\frac{\tilde{E}_h}{E_h})\tilde{v}_h\,=\,0.
  \end{gather*}
  
   Since $E_h$ is in the intermediate regime :
    \begin{itemize}
      \item $W(E_h^{\frac{1}{\gamma}}\cdot)$ converges to the constant function $W(0)$ (uniformly on every compact set),
      \item $\bar{h}\defeq hE_h^{-\frac{2+\gamma}{2\gamma}}$ tends to $0$. 
    \end{itemize}
    We may thus take $\bar{h}$ as a new genuine semiclassical parameter. By construction, we are now working near the energy $1$
    which is non-critical. Assuming we have a spacing of order $\bar{h}$ uniformly for the sequence of potentials
    $z\mapsto z^\gamma W(E_h^{\frac{1}{\gamma}}z)$, we obtain that any eigenvalue $\tilde{E_h}\neq E_h$ must satisfy 
    \[
      |\frac{\tilde{E_h}}{E_h}-1|\,\geq\, c\bar{h}.
    \]
    Thus, we obtain the bound
    \[
      |\tilde{E}_h-E_h|\,\geq\, c h\cdot E^{\frac{\gamma-2}{2\gamma}}.
    \]

    Consequently, we see that Theorem \ref{thm:main} will follow from the usual semiclassical estimates at a non-critical energy
    provided the latter are proven to hold for singular potential and uniformly. This approach is interesting in its own and we will develop
    it after having made the setting precise.

\subsection{Global Assumptions} 

We fix $\gamma>0,\,0<b<c<d$, $\Kcal$ a compact set in $C^\infty([0,d] ; \R)$ equipped with its Fr\'echet
  topology and $K$ a compact set in $(0,+\infty)$. We denote by $\Vcal$ the set of potentials such that the following assumptions hold.
 
\begin{assu}\label{assu:nc}\hfill \\ 
\begin{itemize}
\item The conditions on $V$ from Assumptions \ref{assu:bot} hold.
\item The restriction of $W$ to $[0,d]$ belongs to $\Kcal$.
\item The following estimates hold
\begin{gather*}
\forall (V,E)\in \Vcal \times K,~\forall x\in [0,b],~E-V(x)\, >\, 0, \\
\forall (V,E)\in \Vcal \times K,~\forall x\in [c,d],~V(x)-E\, >\, 0.\\
\end{gather*}
\end{itemize}
\end{assu}

Let us observe that these assumptions imply that
\[
  \forall E\in K,~\forall x\geq d,~V(x)-E \,\geq \,V(d)-E >0,
\]
so that the operator $P_h-E$ is of limit-point type near $\infty$ which allows us to define
$G_h(\cdot\,;\,E)$ for any $E\in K$ and $V\in \Vcal$. Observe that the notation does not reflect the fact
that the function $G$ also depends on $V$.

We want to prove the following theorem.

\begin{Theorem}\label{thm:interregime}
Under the preceding assumptions, there exists $\bfc$ and $h_0$ such that for any 
$h\leq h_0$, for any $V\in \Vcal$ and any $E_h$ eigenvalue of $P_h$:
\[
  E_h\in K \implies d_{h}(E_h)\,\geq\, \bfc h.
\]
\end{Theorem}

The results in Theorem \ref{thm:interregime} will follow from the following proposition.

\begin{prop}\label{prop:minZ}
There exist $\bfc,h_0\,>\,0$ such that 
\[
\forall \ (V,E)\in \Vcal\times K,\,\forall h\leq h_0,~~\left| Z_h\right | \,\geq\,\bfc. 
\]
\end{prop}

The proof of this proposition is somewhat technical and is the main result of this section.  Hence we postpone it until we have discussed how the proof of Theorem \ref{thm:interregime} follows.  

\begin{rem}
  We have renamed $h$ the semiclassical parameter, although, in the scaling argument,
  we use this bound for the rescaled semiclassical parameter $\bar{h}$.
\end{rem}

\begin{proof}[Proof of Theorem \ref{thm:interregime}]
The same computation as that yielding \eqref{Zthetaeq} gives us 
\[
h|Z_h|^2 \dot{\theta}_h \,=\, 1.
\]
We recall that, for any $(V,E)$ in $ \Vcal\times K$, 
\[
Z_h\,\defeq G_h(0\,;\,E)\,+\,ihG'_h(0\,;\,E)\,=\,|Z_h|e^{i\theta_h}.
\]
The claim thus follows from Proposition \ref{prop:minZ}.

\end{proof}

The proof of Proposition \ref{prop:minZ} will proceed by estimating $G_h(\cdot\,;\,E)$ in different regions of the half-line, uniformly
with respect to the potential. To this end, we will need several uniform quantities that we now define.

\begin{rem}
  Observe that the order $h$ spacing at non-critical energies follow from Thm \ref{thm:interregime} by considering $\Vcal \,=\, \{ V\}$.
\end{rem}

\subsection{Uniform Bounds}
 For any $(V,E)\in \Vcal\times K$, the assumptions imply: 
  \begin{itemize}
  \item There is a unique solution $x_E$ to the equation $V(x_E)=E$ (the turning point). 
  \item $[0,b]$ is in the classically allowed region and $(E-V)$ is uniformly bounded below on it.
    \begin{gather}
      \notag \exists \kappa_{o}>0,~\forall \ (V,E)\in \Vcal\times K,~\forall h\leq h_0,\\
      \label{eq:defkappao} \forall x\in [0,b],~E-V(x)\,\geq \kappa_o.
    \end{gather}
    (The $o$ stands for oscillating since, in the classically allowed region, $G_h$ exhibits highly
    oscillating behaviour). 
  \item $[c,+\infty)$ is in the classically not allowed region, and $(V-E)$ is uniformly bounded below on it.
  \begin{gather}
      \notag \exists \kappa_{e}>0,~\forall \ (V,E)\in \Vcal\times K,~\forall h\leq h_0,\\
      \label{eq:defkappae} \forall x\in [c,+\infty),~V(x)-E\,\geq \kappa_e.
    \end{gather}
    (The $e$ stands for exponential).
  \item The turning point $x_E$ always belong to $[b,c]$. Since, on  $[b,c]$, $V'$ is uniformly bounded below,
    the turning point is non-degenerate.
    We also have the following estimate from below:
    \begin{gather*}
      \notag \forall a\,\leq\,b,~\exists \delta_a >0,~\forall (V,E)\in \Vcal\times K,~\forall h\leq h_0,\\
      \label{eq:defdelta} \forall x\in [a,c],~ V'(x)\,\geq\, \delta_a. 
    \end{gather*}
    We will also use the shortcut $\delta\defeq \delta_b$.
  \item Finally, for any $\ell$, $W^{(\ell)}$ is, uniformly on $[0,d]$, bounded above by some $C_\ell$.
  \item If $\gamma$ is an integer, $W^{(\ell)}$ can be replaced by $V^{(\ell)}$ in the latter statement.
  \end{itemize}

  \begin{rem}
    The point $c$ should not be confused with the (different) constant $\bfc$ that appears in the estimates.
  \end{rem}
  
\section{Uniform concentration estimates}
In this section we aim at showing that the mass of $G_h(\cdot \, ;\,E)$ in the classically allowed region
is bounded below uniformly for $(V,E)\in \Vcal \times K$. 
\subsection{In the classically not-allowed region}\label{sec:agmon}
In this section, we prove that the function $G$ is exponentially small in the region $x\geq c$
with constants that are uniform with respect to $V\in \Vcal$ and $E\in K$.
Such exponential estimates are well-known for a fixed pair $(V,E)$. 
We present here a (classical) rudimentary proof that has the advantages of assuming very little on the potential
and of making it very easy to track the constants.

\begin{prop}\label{prop:GNotaR}
  Under the assumptions \ref{assu:nc} and using the preceding notations, for any $(V,E)\in \Vcal \times K$ and for any $\eta>0$,
  we have   
  \begin{equation}\label{eq:GNotaR}
    \forall x,z >0,~~z\,\geq\,x\,\geq\,x_E\,+\,\frac{\eta}{2},~ G(z)\,\leq\, e^{-\frac{\sqrt{\delta \eta}}{2h}(z-x)} G(x), 
  \end{equation}
  in which we recall that
  \[
    \delta \,\defeq\, \inf \{ V'(x),~V\in \Vcal, \, x\in [b,c] \} \,>\,0.
  \]
\end{prop}

\begin{proof}
  First we observe that, for any $x \in [x_E\,+\,\frac{\eta}{2},c],$ we have, uniformly for
  $(V,E)\in \Vcal\times K$, 
  \[
    V(x)-E \,=\,V(x)-V(x_E) \,\geq\, \delta \cdot \frac{\eta}{2}.
  \]
  Since $V$ is increasing on $[c,d],$ the same estimate is true on $[c,d]$ and then on $[d,+\infty)$ since $V(x)\geq V(d)$ on this interval.
  Finally, we obtain: 
  \[
   \forall (V,E)\in \Vcal\times K,~\forall x\,\geq\, x_E+\frac{\eta}{2},~ V(x)-E \geq\, \delta \cdot \frac{\eta}{2}. 
  \]
  
  From the equation
  \[
    -h^2G''\,+\,(V-E)G=0,
  \]
  we thus infer
  \[
    \forall x\geq x_E+\frac{\eta}{2},~(G^2)''(x) \,\geq\,2G(x)G''(x)\,\geq \, \frac{\delta \eta}{h^2}G^2(x). 
  \]
  We set $\omega = \frac{\sqrt{\delta \eta}}{h}$ and, for any $x_E+\frac{\eta}{2}\leq x < y$, we denote by $\phi$ the solution
  to $\phi''\,=\, \omega^2 \phi$ that takes the same values as $G^2$ at $x$ and $y$. Since $G^2-\phi$ vanishes at
  $x$ and $y$ and satisfies 
  $(G^2-\phi)'' \geq \omega^2 (G-\phi)$, a maximum principle argument shows that
  \[
    \forall z\in [x,y],~G^2(z)\,\leq\, \phi(z).
  \]
  By making $\phi$ explicit, we find
  \begin{gather}
    \notag \forall\,x,y,z,~~ x_E+\frac{\eta}{2}\,\leq\, x \,<\,z\,<\, y,~ \\
    \label{est:exp1} G^2(z)\, \leq \,  G^2(x)\frac{\sinh (\omega(y-z))}{\sinh(\omega(y-x))}
    \,+\,G^2(y)\frac{\sinh (\omega(z-x))}{\sinh(\omega(y-x))}.
  \end{gather}
  In this inequality, we fix $x$ and $z$ and integrate with respect to $y$ in $[z+1,z+2]$, we find
  \begin{gather*}
    \forall\,x,z~~x_E+\frac{\eta}{2}\,\leq\,x\,<\,z,\\
~ G^2(z)\,\leq \, G^2(x) \int_{z+1}^{z+2}\frac{\sinh (\omega(y-z))}{\sinh(\omega(y-x))}\, dy\,+\, \int_{z+1}^{z+2}G^2(y)\, dy. 
  \end{gather*}

  It follows that $G^2(z)$ goes to zero when $z$ goes to $\infty$. So we may let $y$ go to $+\infty$ in the estimate \eqref{est:exp1}  and obtain
  \[
    \forall \ x,z,~~x_E+\frac{\eta}{2}\leq x < z,~G^2(z)\,\leq \, G^2(x) e^{\omega (x-z)}.
  \]
  The claim follows by taking the square root, since, by choice, $G$ is positive in the classically not-allowed region. 
\end{proof}

We use this proposition to prove uniform exponential estimates
for the mass of $G$ and for the semiclassical Cauchy data in the classically not-allowed region. 

 \begin{Lemma}
    \begin{gather*}
      \forall \  \eta>0, ~\exists \ \kappa,h_0>0,~\forall \ (V,E)\in \Vcal\times K,~\forall \ h\leq h_0, \\
      \int_{x_E+\eta}^{+\infty} |G_h(x\,;\,E)|^2 \, dx \,\leq e^{-\kappa/h}.
    \end{gather*}
  \end{Lemma}

  \begin{proof}
    We start from the estimate
    \[
    \forall \  x,z,~x_E+\frac{\eta}{2}\leq x < z,~G^2(z)\,\leq \, G^2(x) e^{\omega (x-z)},
  \]
  in which we recall that $\omega \,=\,\frac{\sqrt{\delta \eta}}{h}$.
  For any $x\in [x_E+\frac{\eta}{2},x_E+\eta]$ we integrate this equality over $z\in [x_E+\eta,+\infty)$,
  we find
  \[
    \forall \ x\in [x_E+\frac{\eta}{2},x_E+\eta],\,\int_{x_E+\eta}^{+\infty} G^2(z)\, dz \,\leq\, \frac{1}{\omega}e^{-\omega (x_E+\eta -x)} G^2(x).
  \]
  We may now integrate this inequality over $x\in [x_E+\frac{\eta}{2},x_E+\frac{3\eta}{4}]$. Using that $G$ is $L^2$ normalized, we obtain
  \[
    \frac{\eta}{4} \int_{x_E+\eta}^{+\infty} G^2(z)\, dz \,\leq\, \frac{1}{\omega} e^{-\frac{\omega \eta}{4}}.
  \]
  We obtain finally
  \[
    \int_{x_E+\eta}^{+\infty} G^2(z)\, dz \, \leq \frac{4}{\eta \omega} e^{-\frac{\omega \eta}{4}}.
  \]
  The claim follows if we set $\kappa = \frac{\delta^{\frac{1}{2}}\cdot \eta^{\frac{3}{2}}}{4}$ and choose $h_0$ small enough so that the prefactor is bounded by $1$.
  \end{proof}
  
  We now proceed to give an estimate for the semiclassical Cauchy datum in the classically not-allowed region, using the proposition and the eigenvalue equation for $G$.

\begin{prop}
  There exist $h_1>0$ and a constant  $\kappa_1$ such that, for any $(V,E)\in \Vcal\times K$ and any $h\leq h_1$, we have
  \begin{gather*}
    G(c)\, \leq \,e^{-\frac{\kappa_1}{h}}, \\ 
    |hG'(c)|\, \leq \, e^{-\frac{\kappa_1}{h}}.
  \end{gather*}
\end{prop}

\begin{proof}
  First we observe that, due to compactness, there exists $\bar{\eta}>0$ such that
  \[
    \forall  \ (V,E)\in \Vcal\times K,~x_E\,+\,\bar{\eta} < c.
  \]
  Choosing $\eta =\frac{\bar{\eta}}{2}$ and $h$ small enough, we may thus make sure that
  \[
    \forall \ (V,E)\in \Vcal \times K,~\forall \ h\leq h_1,~ [c-h,c+h]\subset [c-\frac{\bar{\eta}}{2},d] \subset [x_E+\eta,d].
  \]

  Using Proposition \ref{prop:GNotaR}, we thus obtain that
  \[
    \forall \ x\in [x_E+\frac{\eta}{2},x_E+\frac{3\eta}{4}],~\forall \ z\in [c-h,c+h],~ G^2(z)\,\leq \, e^{-\delta^{\frac{1}{2}}\eta^{\frac{3}{2}}/4h} G^2(x),
  \]
  since $(z-x)\geq \frac{\eta}{4}$ for this range of values of $x$ and $z$.
  Integrating with respect to $x$ and taking the square-root we find:
  \[
    \forall \ z\in [c-h,c+h],~ |G(z)| \, \leq \sqrt{\frac{4}{\eta}}e^{-\delta^{\frac{1}{2}}\eta^{\frac{3}{2}}/8h}.
  \]
  This gives the result if we take $\kappa<\delta^{\frac{1}{2}}\eta^{\frac{3}{2}}/8$ and $h$ small enough so that
  $$ \sqrt{\frac{4}{\eta}}e^{-(\delta^{\frac{1}{2}}\eta^{\frac{3}{2}}-\kappa)/8h} \leq 1.$$
  
  Setting
  \[
    M =\sup \{ V(x)-E,~(V,E)\in \Vcal\times K,~x\in [b,d]\},
  \]
  which is finite by compactness, and using the eigenvalue equation, we also have
  \[
    \forall \ z\in [c-h,c+h],~ |h^2G''(z)| \, \leq M\cdot\sqrt{\frac{4}{\eta}}e^{-\delta^{\frac{1}{2}}\eta^{\frac{3}{2}}/8h}.
  \]

  Using Taylor-Lagrange expansions, there exist $\theta_-\in [c-h,c]$ and $\theta_+\in [c,c+h]$ such that
  \[
    \begin{split}
      G(c-h) &\,=\, G(c)\,-h \cdot G'(c) \,+\,\frac{h^2}{2}G''(\theta_-), \\
      G(c+h) &\,=\, G(c)\,+h \cdot G'(c) \,+\,\frac{h^2}{2}G''(\theta_+).\\
    \end{split}
  \]
  By combining these two equations, we obtain
  \[
    |hG'(c)|\, \leq\, \frac{|G(c-h)|+G(c+h)|}{2}\,+\,\frac{1}{4}(h^2|G''(\theta_+)|+h^2|G''(\theta_-)|).
  \]
 
  It then follows from the preceding estimate that there exist some constant $C$ such that
  \[
    |hG'(C)|\leq Ce^{-\delta^{\frac{1}{2}}\eta^{\frac{3}{2}}/8h}.
  \]
  The claim follows by taking the same $\kappa$ as above and a smaller $h_1$ if needed.
\end{proof}

\begin{rem}
Arguing similarly, we could get an estimate replacing $c$ by any $x\in [c,d)$.
\end{rem}

A consequence of this estimate is that the Cauchy data of $G$ at $c$ is exponentially small uniformly for
$(V,E)\in \Vcal\times K$. More precisely, setting $Z_h(\cdot\,;\,E)\,=\, G_h(\cdot \,;\,E)\,+\,ihG'(\cdot\,;\,E)$, we have 
\begin{gather}
  \nonumber \exists \ h_0,~ \kappa>0,~\forall  \ h\leq h_0,~\forall \ (V,E)\in \Vcal\times K,\\
\label{eq:CDNaR}  \left| Z_h(c\,;\,E)\right| \,\leq e^{-\frac{\kappa}{h}}.
\end{gather}

The latter estimates allow control of $G$ in the classically forbidden region.
We will see that, in the classically allowed region, WKB expansions will also provide us with enough control.
It thus remains to address the turning point. There are several ways to do so (using a Maslov or a Airy Ansatz for instance
\cite{Maslov72, Yafaev11}).
We have chosen a semiclassical measure approach since we think it is a nice generalization of the usual theory.  

\subsection{Semiclassical measures for families of potential}\label{sec:measures}
Let $(V_h,E_h)_{h\geq 0}$ be a family in $\Vcal\times K$. For each smooth observable $a$ that is compactly supported in
$(0,d)\times \R$, we define
\[
  \mu_h(a)\,\defeq \,\langle \mathrm{Op}_h(a) G_h,G_h\rangle,
\]
where $\mathrm{Op}_h$ is some semiclassical quantization procedure (see \cite{zworski2012semiclassical} for instance).
A standard argument shows that, up to extracting a subsequence, there exists a limiting measure $\mu_0$.
Using compactness, we may extract again and assume that $V_h$ converges to $V_0$ and $E_h$ converges to $E_0$.

We then have the following proposition that generalizes the known results when the potential is fixed.

\begin{prop}
  Under the preceding assumptions, the support of the semiclassical measure $\mu_0$ is a subset of the energy surface
  \[
    \big \{ \xi^2+V_0(x)\,=\,E_0 , (x,\xi)\in (0,d)\times \R \big \}.
  \]
  The measure $\mu_0$ is invariant by the hamiltonian flow of $p_0(x,\xi)\,\defeq \xi^2+V_0(x).$ 
\end{prop}
\begin{proof}
  We follow the standard proofs. For the support property, we need to show that if $a$ vanishes on a neighbourhood of the energy surface,
  then
  \[
    \mu_h(a) \tv{h}{0} 0.
  \]
  We denote by $P^0_h$ the operator 
  \[
    P^0_h u =-h^2 u''+V_0(x)u.
  \]
  
  We write
  \[
    \begin{split}
      \mu_h(a)\,&=\,\langle \mathrm{Op}_h(a) G_h,G_h\rangle\\
      &=\,\langle \Op_h(\frac{a}{p_0-E_0})(P^0_h-E_0)G_h,G_h\rangle +o(1)\\
      &=\,\langle \Op_h(\frac{a}{p_0-E_0})(V_0-V_h+E_0-E_h)G_h,G_h\rangle +o(1)\\
      &\tv{h}{0} 0,
    \end{split}
  \]
  where we have used that $a$ vanishes on the energy surface so that $\frac{a}{p_0-E_0}$ is smooth with compact support, and
  in the latter stage the fact that $(V_0-V_h+E_0-E_h)$ converges to $0$ on $[0,d]$ and $G_h$ is exponentially small
  on $[d,+\infty)$.

  For the invariance property, we write 
\[
      \begin{split}
        \frac{h}{i}\big[(\langle \mathrm{Op}_h(\left\{p_0,a\right\})G_h,G_h\rangle +o(1)\big]&\,=\,\langle [P_h^0,\mathrm{Op}_h(a)]G_h,G_h\rangle\\
        &\,=\,\langle [P_h,\mathrm{Op}_h(a)]G_h,G_h\rangle\,\\
        & ~~+\,\langle [V_h-V_0,\mathrm{Op}_h(a)]G_h,G_h\rangle \\
        &\,=\,\frac{h}{i}\big[(\langle \mathrm{Op}_h(\left\{V_h-V_0,a\right\})G_h,G_h\rangle\,+\,o(1)\big]) .
      \end{split}
    \]
    We now use the fact that the norm of a pseudodifferential operator on $L^2$ depends on the uniform norm of a finite number of
    derivatives of the symbol and that $\left\{V_h-V_0,a\right\}$ and all its derivatives converge uniformly to $0$ on the support of
    $a$.
  \end{proof}

  The semiclassical measure can be extended to symbols that are not compactly supported in $\xi$, in particular to symbols that only depend
  on $x$.
  
  In dimension $1$, $\mu_0$ is thus determined up to a factor. More precisely, according to the assumptions, there exists $\phi_0$ defined by
\[
\xi^2\,+\,V(\phi_0(\xi))\,=\,E_0
\]
and there exists $c$ such that $\mu_0=c\nu$ where $\nu$ is defined by
  \[
    \nu(a)\,=\,\int a(\phi_0(\xi),\xi)\, \frac{d\xi}{V_0'(\phi_0(\xi))}.
  \]

For a smooth function $\chi$ whose support is a subset of $(0,x_{E_0})$, we have the alternative expression:
\[
\nu(\chi)\,=\, \int \chi(x)\frac{dx}{\sqrt{E-V(x)}}.
\]

Using the semiclassical measure, we obtain that the mass of $G_h$ is uniformly bounded below in the 
classically allowed region $V(x)\leq E$.

\begin{prop}\label{prop:massin0b}
There exists positive constants $\bfc$ and $h_0$ such that 
\begin{equation}\label{eq:massin0b}
\forall \ h\leq h_0,~\forall \ (V,E)\in \Vcal\times K,~~\int_0^b \left| G_h(x\,;\,E)\right|^2 dx \,\geq\, \bfc.
\end{equation}
\end{prop} 

\begin{proof}
The proof is a typical application of using semiclassical measures to prove (non-)concentration estimates. 
By contradiction, we assume that the estimate \eqref{eq:massin0b} does not hold. We can thus find a sequence $(V_h,E_h)$  with 
$h$ going to $0$ such that 
\begin{equation}\label{eq:assum}
\int_0^b \left| G_h(x\,;\,E)\right|^2 dx \tv{h}{0} 0.
\end{equation}
Using compactness, we may first extract subsequences and also assume that  $(V_h,E_h)$ tends to a limiting $(V_0,E_0)$. 
We then extract a subsequence again to obtain a semiclassical measure $\mu_0$. The preceding argument implies that there exists 
$\lambda\geq 0$ such that $\mu_0 = \lambda \nu$.  
Next, we observe that the assumption \eqref{eq:assum} implies that $\lambda=0$. Indeed, for any non-negative function $\chi$ 
that has compact support in $(0,b)$ and that is bounded above by $1$ we have  
\[
\left \langle \mathrm{Op}_h(\chi)G_h, G_h \right \rangle \,=\, \,\int_0^b \chi(x)\left| G_h(x\,;\,E)\right|^2 dx \,
\leq\,\int_0^b \left| G_h(x\,;\,E)\right|^2 dx.
\]
It follows that 
\[
\lambda \int_0^b \chi(x)\,\frac{dx}{\sqrt{E-V(x)}}\,=\,0,
\]
and hence $\lambda=0$.

By choosing an appropriate symbol, this implies that for any closed interval $[x_0,x_1]\subset (0,+\infty)$, we have
\[
\int_{x_0}^{x_1} |G_h(x\,;\,E)|^2\, dx \tv{h}{0} 0. 
\] 

Setting $x_0=b$, summing  and using \eqref{eq:assum}, we obtain that, for any $x_1>b$
\[
\int_{0}^{x_1} |G_h(x\,;\,E)|^2\, dx \tv{h}{0} 0. 
\]
Since $G_h$ is normalized, this implies that the mass of $G_h$ escapes to $+\infty$ but this is in contradiction with the 
estimates in the classically not allowed region.  
\end{proof}

\subsection{In the classically allowed region}\label{sec:WKB}
We now work on $[0,b]$. In this interval, we know that $E-V$ is uniformly bounded from below so that we can
perform WKB approximation of solutions. For the estimate we are looking for only a first order
WKB approximation is needed, but the lack of smoothness at $x=0$ creates small additional complications.
In particular, we will first make the assumption that $\gamma \geq 1$ and then explain how to modify
the proof for $\gamma \in (0,1)$.

\begin{rem}
  We actually conjecture that the following full asymptotic expansion for $Z_h$ holds: 
  \[
Z_h\,=\,\sum_{\substack{m,n\geq 0,\\~m+n\geq 1}} a_{m,n}h^{m\gamma\,+\,n}.
\]
The leading term in that expansion is thus $h^\gamma$ if $\gamma \in (0,1)$ and $h$
if $\gamma \geq 0$. This also explains the two cases.  Proving such a uniform expansion 
will be a topic of future work and is not required to the proof of the results contained here. 
\end{rem}

Let $(V,E)$ be in $\Vcal\times K$ and $G_h(\cdot \,;\,E)$ be defined as before.
We define the functions $S,\,a,\, \phi_{\pm}$  on $[0,b]$ by 
\[
  \begin{split}
    S(x)& =\, \int_0^x \sqrt{E-V(y)}\,dy ,\\ 
    a(x)& =\, (E-V(x))^{-\frac{1}{4}},\\
    \phi_{\pm}(x)&=\,a(x)e^{\pm\frac{i}{h}S(x)}.
  \end{split}
\]

A straightforward computation yields
\[
  -h^2 \phi_{\pm}'' \,+\,(V-E)\phi_{\pm}\,=\,h^2\cdot r\phi_{\pm},
\]
where we have set $r\,\defeq\,-\displaystyle \frac{a''}{a}.$

This computation implies that, on $[0,b]$, $\phi_\pm$ is a basis of solutions to the equation 
\[
-h^2 y''\,+\,(V-E-h^2r)y\,=\,0.
\]
 
Let $u$ be a solution to
\[
  -h^2u''\,+\,(V-E)u\,=\,0.
\]
The classical method consists in saying that $u$ is a solution to the former equation with an inhomogeneous term 
that reads $-h^2ru$ and then in applying the variation of constants method. We find that there exists
constants $\alpha_{\pm}$ such that, for all $x\in (0,b]$, we have 
\[
  u(x)\,=\, \alpha_+\phi_+(x)\,+\,\alpha_-\phi_-(x)\,-\,
  \frac{h}{2i}\int_x^b r(y)u(y)\left[\phi_-(y)\phi_+(x)-\phi_+(y)\phi_-(x) \right ] \,dy.
\]
We define the operator $L_h$ by 
\[
  L_h[u](x)\,=\,\frac{h}{2i}\int_x^b r(y)u(y)\left[\phi_-(y)\phi_+(x)-\phi_+(y)\phi_-(x) \right ] \,dy.
\]
so that the preceding equation rewrites 
\[
  \left( \mathrm{id}\,+\,L_h \right ) [u]\,=\,\alpha_+\phi_+\,+\,\alpha_-\phi_-. 
\]
The operator $L_h$ is easily seen to be linear from $C^0([0,b];\C)$ into itself.

Using the compactness of $\Kcal$ and $K$, there exist $C_1$ and $C_2$ such that, for all $(V,E)\in \Vcal\times K$
and all $y\in [O,b]$:
\[
  \begin{split}
    |r(y)|&\leq\, C_1 y^\rho, \\
    |a(y)|&\leq\, C_2,
  \end{split}
\]
where $\rho = \gamma-2$ if $\gamma \in (0,2)\setminus \{1\}$ and $\rho=0$ if $\gamma = 1$ or $\gamma \geq 2$.

\begin{rem}
In the sequel we will denote by $C$ a generic constant that is uniform for 
$(V,E)$ in $\Vcal\times K$. Observe that this constant may change from one line to the other. 
\end{rem}

We obtain that, for all $(V,E)\in \Vcal\times K$,
\[
  \forall \ x\in (0,b),~~\left| L_h[u](x) \right | \leq \, C\cdot h \cdot \int_x^b y^{\rho}\,dy \cdot \|u\|_{C^0([0,b])}.
\]

If $\gamma \geq 1$ then the integral on the right is convergent and we obtain that the operator norm of $L_h$
is (uniformly w.r.t. $(V,E)$) bounded by $C\cdot h$. 

\begin{prop}\label{prop:firstordergammalarge}
  Let $\gamma \geq 1$ then there exists a constant $C$ that is uniform with respect to $(V,E)\in \Vcal\times K$
  and $h_0$ such that, for any $h\leq h_0$ there exists $\alpha_\pm(h)$ such that
  \[
    \begin{split}
      \| G_h - \alpha_+\phi_+-\alpha_- \phi_- \|_{C^0([0,b])}\,&\leq \, Ch,  \\
      \| hG_h' - \alpha_+h\phi_+'-\alpha_- h\phi_-' \|_{C^0([0,b])}\,&\leq \, Ch. \\
    \end{split}
  \]
  
\end{prop}

\begin{proof}
  According to the previous computation, there exists $\alpha_+$ and $\alpha_-$ so that
  \[
  \left( \mathrm{id}\,+\,L_h \right ) [G]\,=\,\alpha_+\phi_+\,+\,\alpha_-\phi_-,
\]
  and a uniform $C$ such that
  \[
    \| L_h\|_{\Lcal(C^0([0,b]))} \,\leq \,C\cdot h.
  \]
  We choose $h_0$ so that $C\cdot h_0\,<\,1$. It follows that $\mathrm{id}\,+\,L_h$ is invertible and
  \[
    \Big \| \big( \mathrm{id}\,+\,L_h\big)^{-1}-\mathrm{id} \Big \|_{\Lcal(C^0([0,b]))}\,\leq\, C\cdot h.
  \]
  The first estimate on $G_h$ follows. For the second one, we first observe that
  \[
    G'(x)\,=\, \alpha_+\phi_+'(x)\,+\,\alpha_-\phi_-'(x)\,-\,
  \frac{h}{2i}\int_x^b r(y)G(y)\left[\phi_-(y)\phi_+'(x)-\phi_+(y)\phi_-'(x) \right ] \,dy.
  \]
  The integral is then uniformly bounded since $G$ and $h\phi_\pm'$ are bounded in $C^0$ (recall that $\gamma \geq 1$)
  and $r$ is integrable.
\end{proof}

\begin{cor}
  There exist uniform constant $m_1, M_1, m_2, M_2$ so that
  \begin{gather*}
    m_1( |\alpha_+|^2\,+\,|\alpha_-|^2)\,\leq\, \big| G(0\,;\,E)\,+\,ihG'(0\,;\,E) \big|^2 \,\leq\, M_1^2( |\alpha_+|^2\,+\,|\alpha_-|^2), \\
    m_2( |\alpha_+|^2\,+\,|\alpha_-|^2)\,\leq\, \int_0^b |G(x\,;\,E)|^2\, dx \,\leq\, M_2^2( |\alpha_+|^2\,+\,|\alpha_-|^2). \\
  \end{gather*}
\end{cor}

\begin{proof}
We denote by $\alp\,=\,{}^t(\alpha_+,\alpha_-)$ the (column)-vector in $\C^2$ and by 
{$|\alpha|_{\C^2}\defeq (|\alpha_+|^2\,+\,|\alpha_-|^2)^{\frac{1}{2}}$} its norm.
Starting from the expressions in Proposition \ref{prop:firstordergammalarge}, 
we first observe that 
\[
\| \alpha_+\phi_+\,+\alpha_-\phi_-\|_{L^2([0,b])}\,=\, |\alpha|_{\C^2} \left(\int_0^b \frac{dx}{\sqrt{E-V(x)}}\,+\, O(h)\right)^{\frac{1}{2}},
\]
where the $O$ is uniform in $\Vcal\times K$. Indeed, using an integration by parts, the fact that $\gamma \geq 1$ and 
compactness to obtain uniform estimate, we see that the cross-terms give a $O(h)$ contribution.

  Using the triangle inequality then yields 
  \[
    \left(\int_0^b |G(x\,;\,E)|^2\, dx\right)^{\frac{1}{2}}\,=\, |\alpha|_{\C^2}\left(\int_0^b \frac{dx}{\sqrt{E-V(x)}}\,+\,O(h)\right)^{\frac{1}{2}}
\,+\,O(h)
  \]
  in which both $O$ are uniform with respect to $(V,E)\in \Vcal\times K$.
  Since 
  \[ \int_0^b |G(x\,;\,E)|^2\, dx \geq c>0,
  \] 
  and $\int_0^b \frac{dx}{\sqrt{E-V(x)}} \geq c'>0$, both $O$ term
    can be absorbed and we obtain the second line.
    The first line follows using the approximation on $G$ and $hG'$ and the fact that a uniform $O(h)$ term can be absorbed by
    $|\alpha_+|^2\,+\,|\alpha_-|^2$.
  \end{proof}

  Combining the two estimates, and the fact that $\int_0^b |G(x\,;\,E)|^2\, dx$ is uniformly bounded away from $0$, we obtain the proof of Proposition \ref{prop:minZ}.  It remains to address the case $\gamma \in (0,1)$.  

\subsection{When $\gamma \in (0,1)$}  
The problem when $\gamma \in (0,1)$ is that
$y\mapsto y^{\gamma-2}$ is no longer integrable near $0$, so we cannot work directly on $[0,b]$.
It is standard in matching problems that we need to introduce an intermediate point $x_h$ and use different approximations
on $[0,x_h]$ and on $[x_h,b]$. It turns out that we can choose $x_h=h$.

We define the operator $L_h$ as before. Its operator norm  in $\Lcal(C^0([h,b]))$ is bounded above
(uniformly) by
\[
  C\cdot h \cdot \int_h^{b} y^{\gamma-2}\, dy,
\]
so that there exists a uniform $C$ such that 
\[
  \| L_h \| \,\leq\, C\cdot h^\gamma. 
\]

The same proof as above yields the following proposition.

\begin{prop}
  Let $\gamma \in (0,1)$ then there exists a constant $C$ that is uniform with respect to $(V,E)\in \Vcal\times K$
  and $h_0$ such that, for any $h\leq h_0$ there exists $\alpha_\pm$ (that depend on $h$) such that
  \[
    \begin{split}
      \| G_h - \alpha_+\phi_+-\alpha_- \phi_- \|_{C^0([h,b])}\,&\leq \, Ch^\gamma, \\
      \| hG_h' - \alpha_+h\phi_+'-\alpha_- h\phi_-' \|_{C^0([h,b])}\,&\leq \, Ch^\gamma .
    \end{split}
  \]
  
\end{prop}

On $[0,h]$, we follow the same strategy but we take as a basis of pseudosolutions the functions $\psi_{\pm}$ defined by
\[
  \psi_{\pm}(x)\,=\, E^{-\frac{1}{4}}e^{\pm\frac{i\sqrt{E}}{h}x}.
\]
This is equivalent to treating the term $V u$ in the equation as some inhomogeneous term.

By following the same method, we obtain the proposition.

\begin{prop}
  Let $\gamma \in (0,1)$ then there exists a constant $C$ that is uniform with respect to $(V,E)\in \Vcal\times K$
  and $h_0$ such that, for any $h\leq h_0$ there exists $\beta_\pm$ (that depend on $h$) such that
  \[
    \begin{split}
      \| G_h - \beta_+\psi_+-\beta_- \psi_- \|_{C^0([0,h])}\,&\leq \, Ch^\gamma, \\
      \| hG_h' - \beta_+\psi_+'-\beta_- h\psi_-' \|_{C^0([h,b])}\,&\leq \, Ch^\gamma .
    \end{split}
  \]
\end{prop}

Using the former proposition we obtain
\[
  \left\{
    \begin{array}{ccc}
      G(h)&=& \alpha_+ (\phi_+(h)\,+\,O(h^\gamma))\,+\,\alpha_-(\phi_-(h)\,+\,O(h^\gamma))\,+\,O(h^\gamma), \\
      hG'(h)&=& \alpha_+ (h\phi_+'(h)\,+\,O(h^\gamma))\,+\,\alpha_-(h\phi_-'(h)\,+\,O(h^\gamma))\,+\,O(h^\gamma)\\
    \end{array}
  \right .
\]
and using the latter proposition, we obtain
\[
  \left\{
    \begin{array}{ccc}
      G(h)&=& \alpha_+ \psi_+(h)\,+\,\alpha_-\psi_-(h)\,+\,O(h^\gamma) , \\
      hG'(h)&=& \alpha_+ h\psi_+'(h)\,+\,\alpha_- h\psi_-'(h)\,+\,O(h^\gamma).\\
    \end{array}
  \right .
\]
We now observe that
\[
  S(h)\,=\, h(\sqrt{E}\,+\,O(h^\gamma)),~~a(h)\,=\,E^{-\frac{1}{4}}\,+\,O(h^\gamma),~~ha'(h)\,=\, O(h^\gamma) 
\]
so that $\phi_{\pm}(h)\,=\,\psi_{\pm}(h)\,+O(h^\gamma)$ and $h\phi'_{\pm}(h)\,=\,h\psi_{\pm}'(h)+O(h^\gamma)$.
We compute 
\[
  \left |
    \begin{array}{cc}
      \psi_+(h) & \psi_-(h) \\
      h\psi_+'(h) & h\psi_-(h)\\
    \end{array}
  \right |
  \,=\, -2i.
\]
Since this determinant is uniformly bounded away from $0$ and the coefficients of the corresponding matrix
are uniformly bounded above, we deduce that
\[
  \alpha_\pm\,=\,\beta_\pm \,+\,O(h^\gamma).
\]

We now estimate the norms over $[0,h]$ and $[h,b]$:
\[
  \begin{split}
     \| G(x\,;\,E)\|_{L^2([0,h])}&=\, |\beta|_{\C^2}\left[\frac{h}{\sqrt{E}}(1+O(h^{\gamma}))\right]^{\frac{1}{2}}\,+\,O(h^{\gamma+ \frac{1}{2}}), \\
    \|G_h(x\,;\,E)\|_{L^2([h,b])} &=\, |\alpha|_{\C^2}\left[\int_{h}^b \frac{dx}{\sqrt{E-V(x)}}\,+\,O(h)\right]^{\frac{1}{2}}\,+\,O(h^\gamma). 
  \end{split}
\]
Adding these two equalities, and using the fact that $\alpha_\pm\,=\, \beta_\pm\,+\,O(h^\gamma)$ and that
$\int_0^b |G_h(x\,;\,E)|^2 dx$ is uniformly bounded away from $0$, we obtain that
\[
  \begin{split}
    \int_0^b |G_h(x\,;\,E)|^2 dx &= \big(|\alpha_+|^2\,+\,|\alpha_-|^2\big)\int_{0}^b \frac{dx}{\sqrt{E-V(x)}}\,+\,O(h^\gamma)\\
    &= \big(|\beta_+|^2\,+\,|\beta_-|^2\big)\int_{0}^b \frac{dx}{\sqrt{E-V(x)}}\,+\,O(h^\gamma).
  \end{split}
\]

Remarking that
\[
  |Z_h|^2 \asymp \big(|\beta_+|^2\,+\,|\beta_-|^2\big)
\]
completes the proof of Proposition \ref{prop:minZ}.

\section{Proof of Theorem \ref{thm:main}}
To prove part $(1)$, let $M < \liminf_{x\rightarrow +\infty} V(x)$, then 
the part of the spectrum of $P_h$ below $M$ is discrete as follows from the fact that the set 
\[
  \big \{ u\in H^1,~~\|u\|_{L^2}\leq 1,~~
  h^2\int_0^{+\infty} |u'(x)|^2 \, dx \,+\,\int_0^{+\infty} V(x) |u(x)|^2\, dx \, \leq M \| u\|^2_{L^2} \big \}
\]
is relatively compact in $L^2$. 

In order to prove part $(2)$, we argue by contradiction. If the estimate is not true then we can find two distincts eigenvalues
$E_h$ and $\tilde{E_h}$ such that 
\[
  d_h(E) \,=o(h\cdot E^{\frac{\gamma-2}{2\gamma}}).
\]
We may suppose that $E_h$ has a limit $E_0$ and we have three cases to study.
\begin{itemize}
\item $E_0=0$ and there exists $M$ such that $E_h\leq Mh^{\frac{2\gamma}{\gamma +2}}$. We obtain a contradiction using the estimate
  in the bottom of the well regime.
\item $E_0=0$ and $h^{-\frac{2\gamma}{\gamma+2}}E_h \rightarrow +\infty$. We make the energy-dependent scaling and obtain a contradiction
  with the estimate in the intermediate regime.
  \item $E_0>0$. We obtain a contradiction with the non-critical energy regime.    
\end{itemize}

\bibliographystyle{alpha}
\bibliography{MMT-bib1}

\def\cprime{$'$} \def\cftil#1{\ifmmode\setbox7\hbox{$\accent"5E#1$}\else
  \setbox7\hbox{\accent"5E#1}\penalty 10000\relax\fi\raise 1\ht7
  \hbox{\lower1.15ex\hbox to 1\wd7{\hss\accent"7E\hss}}\penalty 10000
  \hskip-1\wd7\penalty 10000\box7}
\begin{thebibliography}{RS78b}

\bibitem[Ber82]{Berry82}
Michael~V Berry.
\newblock Semiclassically weak reflections above analytic and non-analytic
  potential barriers.
\newblock {\em Journal of Physics A: Mathematical and General}, 15(12):3693,
  1982.

\bibitem[BP19]{Bony_Popoff19}
Jean-Francois Bony and Nicolas Popoff.
\newblock Low-lying eigenvalues of semiclassical schr{\"o}dinger operator with
  degenerate wells.
\newblock {\em Asymptotic Analysis}, 112(1-2):23--36, 2019.

\bibitem[BS12]{BirSol87}
Michael~Sh Birman and Michael~Z Solomjak.
\newblock {\em Spectral theory of self-adjoint operators in Hilbert space},
  volume~5.
\newblock Springer Science \& Business Media, 2012.

\bibitem[Chr15]{Chr_AIF_15}
Hans Christianson.
\newblock Unique continuation for quasimodes on surfaces of revolution:
  Rotationally invariant neighbourhoods.
\newblock In {\em Annales de l'Institut Fourier}, volume~65, pages 1617--1645,
  2015.

\bibitem[DS99]{DiSj}
Mouez Dimassi and Johannes Sj{\"o}strand.
\newblock {\em Spectral asymptotics in the semi-classical limit}, volume 268 of
  {\em London Mathematical Society Lecture Note Series}.
\newblock Cambridge University Press, Cambridge, 1999.

\bibitem[dV05]{Cdv_BS05}
Y~Colin de~Verdiere.
\newblock Bohr-sommerfeld rules to all orders.
\newblock {\em Ann. Henri Poincar{\'e}}, 6(5):925--936, 2005.

\bibitem[FS09]{friedlander2009spectrum}
Leonid Friedlander and Michael Solomyak.
\newblock On the spectrum of the dirichlet laplacian in a narrow strip.
\newblock {\em Israel journal of mathematics}, 170(1):337--354, 2009.

\bibitem[GW18]{GW}
Oran Gannot and Jared Wunsch.
\newblock Semiclassical diffraction by conormal potential singularities.
\newblock Preprint, arXiv:1806.01813, 2018.

\bibitem[GZ06]{gesztesy2006spectral}
Fritz Gesztesy and Maxim Zinchenko.
\newblock On spectral theory for schr{\"o}dinger operators with strongly
  singular potentials.
\newblock {\em Mathematische Nachrichten}, 279(9-10):1041--1082, 2006.

\bibitem[Hil18]{Hil_Kai}
Luc Hillairet.
\newblock Two applications of {D}irichlet-{N}eumann bracketing.
\newblock In {\em Spectral theory of graphs and of manifolds, {CIMPA} 2016,
  {K}airouan, {T}unisia}, volume~32 of {\em S\'{e}min. Congr.}, pages 249--261.
  Soc. Math. France, Paris, 2018.

\bibitem[HM12]{hillairet2012nonconcentration}
Luc Hillairet and Jeremy Marzuola.
\newblock Nonconcentration in partially rectangular billiards.
\newblock {\em Analysis \& PDE}, 5(4):831--854, 2012.

\bibitem[H{\"o}r80]{Hormander3}
Lars H{\"o}rmander.
\newblock {\em The Analysis of Linear Partial Differential Operators III}.
\newblock Springer, 1980.

\bibitem[H{\"o}r03]{Hor-v1}
Lars H{\"o}rmander.
\newblock {\em The analysis of linear partial differential operators. {I}}.
\newblock Classics in Mathematics. Springer-Verlag, Berlin, 2003.
\newblock Distribution theory and Fourier analysis, Reprint of the second
  (1990) edition [Springer, Berlin; MR1065993 (91m:35001a)].

\bibitem[H{\"o}r05]{Hor-v2}
Lars H{\"o}rmander.
\newblock {\em The analysis of linear partial differential operators. {II}}.
\newblock Classics in Mathematics. Springer-Verlag, Berlin, 2005.
\newblock Differential operators with constant coefficients, Reprint of the
  1983 original.

\bibitem[H{\"o}r09]{Hormander4}
Lars H{\"o}rmander.
\newblock {\em The analysis of linear partial differential operators. {IV}}.
\newblock Classics in Mathematics. Springer-Verlag, Berlin, 2009.
\newblock Fourier integral operators, Reprint of the 1994 edition.

\bibitem[HR83]{helffer1983calcul}
Bernard Helffer and Didier Robert.
\newblock Calcul fonctionnel par la transformation de mellin et op{\'e}rateurs
  admissibles.
\newblock {\em Journal of functional analysis}, 53(3):246--268, 1983.

\bibitem[HS84]{HeSjo_mwI84}
Bernard Helffer and Johannes Sjostrand.
\newblock Multiple wells in the semi-classical limit i.
\newblock {\em Communications in Partial Differential Equations},
  9(4):337--408, 1984.

\bibitem[LR79]{LaiRobert79}
Pham~The {Lai} and D.~{Robert}.
\newblock {Valeurs propres d'une classe d'\'equations diff\'erentielles
  singulieres sur une demi-droite}.
\newblock {\em {Ann. Sc. Norm. Super. Pisa, Cl. Sci., IV. Ser.}}, 6:335--366,
  1979.

\bibitem[MA72]{Maslov72}
Viktor~Pavlovich Maslov and Vladimir~I Arnol'd.
\newblock {\em Th{\'e}orie des perturbations et m{\'e}thodes asymptotiques}.
\newblock Dunod, 1972.

\bibitem[MR88]{Mart_Rou88}
Andr\'{e} Martinez and Michel Rouleux.
\newblock Effet tunnel entre puits d\'{e}g\'{e}n\'{e}r\'{e}s.
\newblock {\em Comm. Partial Differential Equations}, 13(9):1157--1187, 1988.

\bibitem[OB78]{Bender_Orszag78}
S~Orszag and Carl~M Bender.
\newblock {\em Advanced mathematical methods for scientists and engineers}.
\newblock McGraw-Hill New York, NY, USA, 1978.

\bibitem[Olv74]{olver}
Frank W.~J. Olver.
\newblock {\em Asymptotics and Special Functions}.
\newblock Academic Press, New York-London, 1974.

\bibitem[RS78a]{reed1978methods}
M~Reed and B~Simon.
\newblock Methods in mathematical physics, vol. iv: Analysis of operators,
  1978.

\bibitem[RS78b]{RSv4}
M.~Reed and B.~Simon.
\newblock {\em Methods of Modern Mathematical Physics IV. Analysis of
  Operators}.
\newblock Academic Press, 1978.

\bibitem[Sim83]{Simon_lowlying_83}
Barry Simon.
\newblock Semiclassical analysis of low lying eigenvalues. i. non-degenerate
  minima: Asymptotic expansions.
\newblock In {\em Annales de l'IHP Physique th{\'e}orique}, volume~38, pages
  295--308, 1983.

\bibitem[Sj{\"o}92]{Sjostrand_semi92}
Johannes Sj{\"o}strand.
\newblock Semi-excited states in nondegenerate potential wells.
\newblock {\em Asymptotic analysis}, 6(1):29--43, 1992.

\bibitem[Tes09]{teschl2009mathematical}
Gerald Teschl.
\newblock Mathematical methods in quantum mechanics.
\newblock {\em Graduate Studies in Mathematics}, 99, 2009.

\bibitem[Tit46]{titchmarsh1946eigenfunction}
EC~Titchmarsh.
\newblock Eigenfunction expansions associated with second-order differential
  equations, vol. 1. clarendon, 1946.

\bibitem[Yaf11]{Yafaev11}
Dimitri Yafaev.
\newblock The semiclassical limit of eigenfunctions of the schr{\"o}dinger
  equation and the bohr--sommerfeld quantization condition, revisited.
\newblock {\em St. Petersburg Mathematical Journal}, 22(6):1051--1067, 2011.

\bibitem[Zwo12]{zworski2012semiclassical}
Maciej Zworski.
\newblock {\em Semiclassical analysis}, volume 138.
\newblock American Mathematical Soc., 2012.

\end{thebibliography}

\end{document}